\documentclass[12pt, twoside]{article}
\usepackage{amsmath,amsthm,amssymb}
\usepackage{times}
\usepackage{enumerate}

\def\titlerunning#1{\gdef\titrun{#1}}
\makeatletter
\def\author#1{\gdef\autrun{\def\and{\unskip, }#1}\gdef\@author{#1}}
\def\address#1{{\def\and{\\\hspace*{18pt}}\renewcommand{\thefootnote}{}%
\footnote {#1}}%
\markboth{\autrun}{\titrun}}
\makeatother
\def\email#1{e-mail: #1}
\def\subjclass#1{{\renewcommand{\thefootnote}{}%
\footnote{\emph{Mathematics Subject Classification $(2010)$:} #1}}}
\def\keywords#1{\par\medskip
\noindent\textbf{Keywords.} #1}

\newtheorem{Theorem}{Theorem}[section]
\newtheorem{Lemma}[Theorem]{Lemma}

\newtheorem{Corollary}[Theorem]{Corollary}
\newtheorem{Proposition}[Theorem]{Proposition}\theoremstyle{definition}

\theoremstyle{definition}
\newtheorem{Remark}[Theorem]{Remark}

\newtheorem{Notation}[Theorem]{Notation}

\numberwithin{equation}{section}

\begin{document}


\baselineskip=17pt


\titlerunning{Regularity of stable solutions to semilinear elliptic equations on Riemannian models}

\title{Regularity of stable solutions to semilinear elliptic equations on Riemannian models}

\author{Daniele Castorina
\and
Manel Sanch\'on}

\date{}

\maketitle

\address{D. Castorina: Universit\'a di Roma ``Tor Vergata'',
Dipartimento di Matematica, Via della Ricerca Scientifica, 00133
Roma, Italy; \email{castorin@mat.uniroma2.it}
\and
M. Sanch\'on: Centre de Recerca Matem\`atica and 
Universitat Aut\`onoma de Barcelona, 
Campus de Bellaterra, Edifici C, 08193 Bellaterra (Barcelona), Spain; 
\email{manel.sanchon@gmail.com}}



\begin{abstract}
We consider the reaction-diffusion problem $-\Delta_g u = f(u)$ in $\mathcal{B}_R$ with zero 
Dirichlet boundary condition, posed in a geodesic ball $\mathcal{B}_R$ with radius 
$R$ of a Riemannian model $(M,g)$. This class of Riemannian manifolds includes the 
classical \textit{space forms}, \textit{i.e.}, the Euclidean, elliptic, and 
hyperbolic spaces.
For the class of semistable solutions we prove radial symmetry and monotonicity. 
Furthermore, we establish $L^\infty$, $L^p$, and $W^{1,p}$ estimates which are optimal and do not 
depend on the nonlinearity $f$. As an application, under standard assumptions on the nonlinearity 
$\lambda f(u)$, we prove that the corresponding extremal solution $u^*$ is bounded whenever 
$n\leq9$. To establish the optimality of our regularity results we find the extremal solution 
for some exponential and power nonlinearities using an improved weighted Hardy inequality.

\end{abstract}
\section{Introduction}\label{intro}

This article is concerned with semilinear elliptic reaction-diffusion problems 
on Riemannian manifolds. We are interested in the class of semistable solutions, 
which include local minimizers, minimal solutions, extremal solutions, and also 
certain solutions found between a sub and a supersolution. On any geodesic ball, 
we show that semistable solutions are radially symmetric and decreasing. Then, 
we establish $L^\infty$, $L^p$, and $W^{1,p}$ \textit{a priori} estimates for 
solutions in this class. As an application we obtain sharp regularity results 
for extremal solutions. To show the optimality of our regularity results we find 
the extremal solution for some exponential and power nonlinearities. This will 
follow by using an improved weighted Hardy inequality for radial functions.

We point out that the regularity properties we achieve in this paper represent a 
geometrical extension of the ones carried out by Cabr\'e and Capella in 
\cite{CC05} for the Euclidean case. As in \cite{CC05}, our results do not depend 
on the specific form of the nonlinearity in the reaction term and they show that 
the class of semistable solutions enjoys better regularity properties than general 
solutions.

More specifically, let $f$ be any locally Lipschitz positive nonlinearity 
and consider the following semilinear elliptic problem
\begin{equation}\label{probintro}
\left\{
\begin{array}{rcll}
-\Delta_{g} u &=& f(u)  &\mbox{in}\, \mathcal{B}_R, \\
u &>& 0 &\mbox{in}\, \mathcal{B}_R,\\
u &=& 0 &\mbox{on}\, \partial \mathcal{B}_R,
\end{array}
\right.
\end{equation}
posed on a geodesic ball $\mathcal{B}_R$, with radius $R$, of a \textit{Riemannian model} 
$(M,g)$. That is, a manifold $M$ of dimension $n \geq 2$ admitting a pole $O$ and whose metric 
$g$ is given, in spherical/polar coordinates around $O$, by
\begin{equation}\label{polar}
ds^2 = dr^2 + \psi(r)^2 d\Theta^2 \quad \textrm{ for }r\in(0,R)
\textrm{ and } \Theta \in \mathbb{S}^{n-1},
\end{equation}
where $r$ is the geodesic distance of the point $P=(r,\Theta)$ to the 
pole $O$, $\psi$ is a smooth positive function in $(0,R)$, and 
$d\Theta^2$ is the canonical metric on the unit sphere $\mathbb{S}^{n-1}$.  
A similar setting has been recently considered by Berchio, Ferrero, and Grillo 
\cite{BFG} in order to study stability and qualitative properties of radial solutions
to the Lane-Emden-Fowler equation, where $f(u) = |u|^{m-1} u$ with $m>1$,  
on certain classes of Cartan-Hadamard manifolds with infinite volume and negative sectional curvatures.

Observe that \eqref{polar} defines the metric only away from the origin. 
From \cite{GS} and \cite{KW}, in order to extend in a $C^2$ manner the metric 
$d s^2$ to the whole $\mathbb{R}^n$ it is sufficient to impose the following conditions:
\begin{equation}\label{hyppsi}
\psi(0) = \psi''(0) = 0 \quad\textrm{and}\quad \psi'(0) = 1.
\end{equation}
Important consequences of the above hypotheses \eqref{hyppsi}, as discussed in \cite{GS},
are that on geodesic balls of $M$ the Laplace-Beltrami operator $-\Delta_g$ is uniformly elliptic and its $L^2$ spectrum is bounded away from zero.
 
Our purpose is to study the regularity of semistable solutions of 
\eqref{probintro}. We say that a classical solution $u\in C^2(\mathcal{B}_R)$ 
of \eqref{probintro} is \textit{semistable} if the linearized operator at $u$ 
is nonnegative definite, \textit{i.e.},
\begin{equation}\label{semistability}
\int_{\mathcal{B}_R} |\nabla_{g} \xi|^2 \, dv_{g}
\geq 
\int_{\mathcal{B}_R} f'(u) \xi^2 \, dv_{g} \quad
\textrm{for all } \, \xi \in C^1_{0} (\mathcal{B}_R).
\end{equation}

The following theorem establishes radial symmetry and monotonicity properties 
of semistable classical solutions $u\in C^2(\mathcal{B}_R)$. By a \textit{radially  
symmetric and decreasing function}  $u\in C^2(\mathcal{B}_R)$ we mean a function $u$ 
such that $u=u(r)$, with $r=|x|$, and $u_r(r)=(du/dr)(r)<0$ for all $r\in(0,R)$.

\begin{Theorem}\label{radsemistab}
Let $f$ be a locally Lipschitz positive function. Assume that $\psi\in C^2([0,R])$ 
is positive in $(0,R]$ and satisfies \eqref{hyppsi}.
If $u \in C^2(\mathcal{B}_R)$ is a semistable solution of \eqref{probintro}, 
then it is radially symmetric and decreasing.
\end{Theorem}

The proof of Theorem \ref{radsemistab} makes no use of moving plane arguments 
as usual. Instead, the radial symmetry relies on the fact that, due to the 
semistability, any angular derivative of $u$ would be either a sign changing first 
eigenfunction of the linearized operator at $u$ or identically zero. However, the 
first assertion cannot hold since the first eigenfunction of the linearized operator 
should be positive.
The monotonicity is then a trivial consequence of 
the positivity of the nonlinearity $f$.

Our first main result establishes \textit{a priori} estimates for semistable classical 
solutions of \eqref{probintro}. This result is useful in order to obtain the 
regularity solutions, \textit{a priori} possibly singular, that can be obtained as the limit 
of semistable classical solutions (see for instance the application on minimal 
and extremal solutions below).

\begin{Theorem}\label{thm1}
Assume that $\psi\in C^2([0,R])$ is positive in $(0,R]$ and satisfies \eqref{hyppsi}.
Let $f$ be a locally Lipschitz positive function and  
\begin{equation}\label{crit:exp}
p_0:=\frac{2n}{n- 2\sqrt{n-1}-4}
\quad \textrm{and}\quad
p_1:=\frac{2n}{n- 2\sqrt{n-1}-2}.
\end{equation}

If $u\in C^2(\mathcal{B}_R)$ is a semistable solution of \eqref{probintro}, 
then the following assertions hold:
\begin{itemize}
\item[$(a)$] If $n\leq9$ then there exists a constant $C_{n,\psi}$ depending only on
$n$ and $\psi$ such that 
\begin{equation}\label{Linfty:estimate}
 \Vert u\Vert_{L^\infty(\mathcal{B}_R)} \leq C_{n,\psi} \| u \|_{L^1 (\mathcal{B}_R)}.
\end{equation}

\item[$(b)$] If $n \geq 10$ then there exist constants $C_{n,\psi,p}$ and 
$\overline{C}_{n,\psi,p}$ depending only on $n$, $\psi$, and $p$ such that 
\begin{equation}\label{Lp:estimate}
\| u \|_{L^p(\mathcal{B}_R)}\leq C_{n,\psi,p} \| u \|_{L^1(\mathcal{B}_R)} 
\qquad \textrm{for all }
p<p_0
\end{equation}
and 
\begin{equation}\label{W1p:estimate}
\| u \|_{W^{1,p}(\mathcal{B}_R)}\leq \overline{C}_{n,\psi,p} \| u \|_{L^1(\mathcal{B}_R)} 
\qquad \textrm{for all }
p<p_1.
\end{equation}
\end{itemize}
\end{Theorem}

\begin{Remark}
Note that the denominator of the exponent $p_0$ in \eqref{crit:exp} is positive for 
$n>10$, while it vanishes for $n=10$. This exponent has to be understood as infinity for $n=10$. 

In dimensions $n\leq 9$, every solution, \textit{a priori} possibly singular, which is limit of semistable 
classical solutions is bounded by Theorem~\ref{thm1}~(i), and thus it is in fact a classical solution. 
In this sense, Theorem~\ref{thm1} may be regarded as a result on removable singularities. 
\end{Remark}

Cabr\'e and Capella \cite{CC05} proved Theorem~\ref{thm1} in the Euclidean 
case: $\psi(r)=r$. The proof of our main theorem, as in \cite{CC05}, relies 
essentially on the following key estimate
\begin{equation}\label{keykey}
\int_{0}^{\delta} u_{r}^{2} \psi^{n-1-2\alpha}\,dr 
\leq
C_{n,\alpha,\psi} \| u \|_{L^1 (\mathcal{B}_R)}^2
\end{equation}
for some $\delta\in(0,R)$ and some range of explicit $\alpha$ (see 
Lemma~\ref{lemma:key} below). 
This estimate is obtained by using the radial symmetry of the solution  
and by choosing $\xi=|u_r|\eta$ as a new test function in the 
semistability condition \eqref{semistability}. With this choice, we 
have to be careful in the computations due to the appearance of the
first and second derivatives of $\psi$ (which in the Euclidean case
are identically $1$ and $0$, respectively). As we will see, the general 
assumptions \eqref{hyppsi} on $\psi$ will be enough to prove \eqref{keykey}.

Note that our result applies to the important case of \textit{space forms}, 
\textit{i.e.}, the unique complete and simply connected Riemannian manifold 
$M$ of constant sectional curvature $K_\psi$ given by
\begin{itemize}
	\item the hyperbolic space $\mathbb{H}^n$: $\psi (r) = \sinh r$ and $K_\psi=-1$;
	\item the Euclidean space $\mathbb{R}^n$: $\psi (r) = r$ and $K_\psi=0$;
	\item the elliptic space $\mathbb{S}^n$: $\psi (r) = \sin r$ and $K_\psi=1$.
\end{itemize}

In Theorems~\ref{thm:hyper} and \ref{thm:ellip} below we present explicit extremal 
solutions (which are limit of classical semistable solutions) for some exponential 
and power nonlinearities. These explicit solutions, as in the flat case, show the 
sharpness of the $L^\infty$, $L^p$, and $W^{1,p}$ estimates of Theorem~\ref{thm1} 
in geodesic balls $\mathcal{B}_R$ of the above space forms.  

As main application of Theorem~\ref{thm1}, we consider the following problem
\begin{equation}\label{probla}
\left\{
\begin{array}{rcll}
-\Delta_g u&=&\lambda f(u)& \textrm{ in }\Omega,\\
u& > & 0 & \textrm{ in }\Omega,\\
u& = & 0 & \textrm{ on }\partial \Omega,
\end{array}
\right.
\end{equation}
where $\Omega$ is a smooth bounded domain in $M$, $\lambda>0$, and $f$ is an increasing 
$C^1$ function satisfying $f(0)>0$ and
\begin{equation}\label{superlinear}
\lim_{t\rightarrow+\infty}\frac{f(t)}{t}=+\infty.
\end{equation}

The study of the above nonlinear eigenvalue problem requires to extend 
to the general case of Riemannian models the classical results of 
Crandall and Rabinowitz~\cite{CranRa75} and Brezis \textit{et al.}~\cite{BreCaMarRa96} 
for the Euclidean setting (see also Proposition~5.1 in \cite{CC05}). 
More specifically, since the first eigenvalue of 
$-\Delta_g$ on $\Omega$ is positive (as well as the corresponding eigenfunction) 
and we have a comparison principle for $-\Delta_g$ (since it is uniformly elliptic), 
it is standard to prove that there exists a parameter value $\lambda^*\in(0,+\infty)$ 
such that: if $0<\lambda <\lambda^*$ then \eqref{probla} admits a minimal solution 
$u_\lambda\in C^2 (\overline{\Omega})$, 
while for $\lambda>\lambda^*$ problem \eqref{probla} does not admit any classical solution. 
Here minimal means smaller than any other supersolution of the 
problem. Moreover, we also have that for every $0< \lambda < \lambda^*$ the minimal 
solution $u_\lambda$ is semistable in the sense of \eqref{semistability}. 
These assertions can be obtained as in Proposition~5.1~$(a)$-$(b)$ of \cite{CC05}.

Moreover, the increasing limit of minimal solutions
\begin{equation}\label{extremal}
u^*:=\lim_{\lambda\uparrow \lambda^*}u_\lambda,
\end{equation}
which is well defined by the pointwise increasing property of $u_\lambda$ with 
respect to $\lambda$, becomes a \textit{weak} solution of \eqref{probla} 
for $\lambda=\lambda^*$ in the following sense: $u^*\in L^1(\mathcal{B}_R)$, 
$f(u^*) (R-r) \in L^1(\mathcal{B}_R)$, and 
\begin{equation}\label{veryweak}
-\int_{\mathcal{B}_R} u \Delta_{g} \xi \, dv_{g}= 
\lambda\int_{\mathcal{B}_R} f(u) \xi \, dv_{g} \quad
\textrm{for all } \, \xi \in C^1_{0} (\mathcal{B}_R).
\end{equation}
This solution $u^*$ is called the \emph{extremal solution} of \eqref{probla} for $\lambda=\lambda^*$.
This statement follows as in Proposition~5.1~$(c)$ of \cite{CC05}. 

Applying Theorem~\ref{thm1}~(a) or (b) (depending on the dimension $n$) 
to minimal solutions $u_\lambda$ and letting $\lambda\uparrow\lambda^*$ it is 
straightforward to see that $u^*$ enjoys the same regularity properties 
as the ones stated in Theorem~\ref{thm1}: 

\begin{Corollary}\label{cor1}
Assume that $\psi\in C^2([0,R])$ is positive in $(0,R]$ and satisfies \eqref{hyppsi}.
Let $f$ be a $C^1$ positive and increasing function satisfying 
\eqref{superlinear}. 
Let $u^*\in L^1(\mathcal{B}_R)$ be the extremal solution of 
\eqref{probla} and $p_0$, $p_1$ the exponents defined in \eqref{crit:exp}. 
Then the following assertions hold:
\begin{itemize} 
\item[$(i)$] If $n\leq 9$ then $u^*\in L^\infty(\mathcal{B}_R)$.
\item[$(ii)$] If $n\geq 10$ then $u^*\in L^p(\mathcal{B}_R)\cap W^{1,q}(\mathcal{B}_R)$ for all 
$p<p_0$ and $q<p_1$.
\end{itemize}
\end{Corollary}

As second main result, we obtain the extremal solution for some exponential and power nonlinearities. 
More precisely, given 
\begin{equation}\label{kpsi}
K_\psi
:=
\left\{
\begin{array}{lll}
-1&\textrm{if}&\psi=\sinh,
\\
0&\textrm{if}&\psi={\rm Id},
\\
1&\textrm{if}&\psi=\sin,
\end{array}
\right.
\end{equation}
we consider the following exponential and power nonlinearities:
\begin{equation}\label{exp:nonlinearity}
f_{\rm e}(u) = \frac{e^{u}}{\psi(R)^2}-\frac{n-1}{n-2}K_\psi
\end{equation}
and
\begin{equation}\label{power:nonlinearity}
f_{\rm p}(u) = (u+\psi(R)^{-\frac{2}{m-1}})\left((u+\psi(R)^{-\frac{2}{m-1}})^{m-1}
-\frac{(m-1)n-(m+1)}{(m-1)n-2m}K_\psi\right),
\end{equation}
where $m>1$.

Note that for $\psi(r)=r$ (the Euclidean case) and $R=1$ (the unit ball), we 
recover the classical nonlinearities $e^u$ and $(1+u)^m$ studied in detail by Joseph 
and Lundgren~\cite{JoLund73}, Crandall and Rabinowitz~\cite{CranRa75}, Mignot and 
Puel~\cite{MP80}, and Brezis and V\'azquez~\cite{BreVaz97}. For these nonlinearities
the extremal parameter and the extremal solution of \eqref{probla} are as follows:
\begin{itemize}
\item If $f(u)=e^u$ and $n\geq 10$ then $\lambda^*=2(n-2)$ and $u^*(r)=\log(1/r^2)$.
\item If $f(u)=(1+u)^m$ and 
\begin{equation}\label{dim:powers}
n\geq N(m):=2+4\frac{m}{m-1}+4\sqrt{\frac{m}{m-1}},
\end{equation}
then $\lambda^*=\frac{2}{m-1}\left(n-\frac{2m}{m-1}\right)$ and 
$u^*(r)=r^{-\frac{2}{m-1}}-1$.
\end{itemize}

We extend this result to the hyperbolic and the elliptic spaces. 
In the hyperbolic space we find the extremal parameter and the extremal solution 
of~\eqref{probla} for both nonlinearities (the ones defined in \eqref{exp:nonlinearity} and 
\eqref{power:nonlinearity}) in any geodesic ball.
\begin{Theorem}\label{thm:hyper}
Assume $\psi=\sinh$. Let $f_{\rm e}$ and $f_{\rm p}$ be the nonlinearities defined 
in \eqref{exp:nonlinearity} and \eqref{power:nonlinearity}, respectively, and let $N(m)$ be
defined in \eqref{dim:powers}.
The following assertions hold:

\begin{itemize}
\item[$(i)$] Let $f=f_{\rm e}$. If $n\geq 10$, 
then
$$
\lambda^*=2(n-2)\quad \textrm{and}\quad u^*(r)=-2\log \left(\frac{\sinh(r)}{\sinh(R)}\right).
$$ 

\item[$(ii)$] Let  $f=f_{\rm p}$ with $m>1$. If $n\geq N(m)$
then
$$
\lambda^*=\frac{2}{m-1}\left(n-\frac{2m}{m-1}\right)
\quad\textrm{and}\quad
u^*(r)=\sinh(r)^{-\frac{2}{m-1}}-\sinh(R)^{-\frac{2}{m-1}}.
$$
\end{itemize}
\end{Theorem}

Instead, in the elliptic space we find the extremal parameter and the extremal solution only in 
sufficiently small balls.

\begin{Theorem}\label{thm:ellip}
Assume $\psi=\sin$. Let $f_{\rm e}$ and $f_{\rm p}$ be the nonlinearities defined 
in \eqref{exp:nonlinearity} and \eqref{power:nonlinearity}, respectively, and let $N(m)$ be
defined in \eqref{dim:powers}.
Let 
\begin{equation}\label{R0}
R_0:=\sup\{s\in(0,\pi/2): \frac{\sin^2 s}{(1-\cos s)^2}>n(n-2)\}.
\end{equation}

The following assertions hold:

\begin{itemize}
\item[$(i)$] \hspace{-0.1cm}Let $f=f_{\rm e}$ and  $R_{\rm e}:=\arcsin\left(\sqrt{\frac{n-2}{n-1}}\right)\in(0,\pi/2)$.
If $n\geq 10$ and $R<\min\{R_0,R_{\rm e}\}$, then
$$
\lambda^*=2(n-2)\quad \textrm{and}\quad u^*(r)=-2\log \left(\frac{\sin(r)}{\sin(R)}\right).
$$ 

\item[$(ii)$] \hspace{-0.1cm}Let $f=f_{\rm p}$ with $m>1$ and $R_{\rm p}
:=\arcsin\left(\sqrt{\frac{n-2}{n}}\right)\in(0,\pi/2)$. 
If $n\geq N(m)$ and $R<\min\{R_0,R_{\rm p}\}$ then
$$
\lambda^*=\frac{2}{m-1}\left(n-\frac{2m}{m-1}\right)
\quad\textrm{and}\quad
u^*(r)=\sin(r)^{-\frac{2}{m-1}}-\sin(R)^{-\frac{2}{m-1}}.
$$
\end{itemize}
\end{Theorem}

\begin{Remark}
(i) These examples show the sharpness of our regularity results for any 
geodesic ball in the hyperbolic space and for geodesic balls of small enough radius in 
the elliptic space. 
For the exponential nonlinearity we obtain that the extremal solution 
$u^*(r)=-2\log \left(\psi(r)/\psi(R)\right)$ ---which is limit of semistable 
classical solutions--- is unbounded at the origin if $n \geq 10$. 
This shows the optimality of Theorem~\ref{thm1}~$(a)$. 
Instead, for the power nonlinearity we obtain that the extremal solution
$u^*(r)=\psi(r)^{-\frac{2}{m-1}}-\psi(R)^{-\frac{2}{m-1}}$ belongs 
exactly to the $L^p$ and $W^{1,p}$ spaces stated in Theorem~\ref{thm1}~$(b)$. This shows 
the sharpness of the exponents $p_0$ and $p_1$ defined in \eqref{crit:exp}.

(ii) In Theorem~\ref{thm:ellip}~(i) we make the assumption $R<\min\{R_0,R_{\rm e}\}$. We assume 
$R<R_{\rm e}$ in order to ensure that the exponential nonlinearity defined 
in \eqref{exp:nonlinearity} is positive. Instead, we assume $R<R_0$ in order to have a
Hardy-type inequality (see Proposition~\ref{hardyprop} below). 
The assumptions on $R$ in Theorem~\ref{thm:ellip}~(ii) are set exactly for the same reasons.
\end{Remark}

To prove Theorems~\ref{thm:hyper} and \ref{thm:ellip} we proceed as in \cite{BreVaz97}.
That is, we use the uniqueness of semistable solutions in the energy class $H^1_0(\mathcal{B}_R)$ 
(see Proposition~\ref{uniqueness} below) and the following improved Hardy inequality.

\begin{Proposition}[Improved weighted Hardy inequality]\label{hardyprop}
Assume $n\geq 3$. Let $\psi$ either $\sinh$ or $\sin$, and $K_\psi$ and $R_0$ be defined in 
\eqref{kpsi} and \eqref{R0}, respectively. 
The following inequality holds: 
\begin{equation}\label{improved:hardy}
\int_{0}^{R} \psi^{n-1} \xi_{r}^{2} \, dr 
\geq
\frac{(n-2)^2}{4} \int_{0}^{R}  \psi^{n-1} \frac{\xi^{2}}{\psi^2}   \, dr
+
H_{n,\psi}\int_{0}^{R} \psi^{n-1}\xi^2 \, dr
\end{equation} 
for all radial $\xi \in C^1_0(\mathcal{B}_R)$, where 
\begin{equation}\label{Hardy:ctant}
H_{n,\psi}=\frac{1}{4}\left((\sup_{(0,R)}(\phi/\psi))^{-2}-n(n-2)K_\psi\right)
\end{equation}
and $\phi(r):=\int_0^r\psi(s)\,ds$ for all $r\in(0,R)$. 

If in addition $R<R_0$ when $\psi=\sin$, then $H_{n,\psi}>0$. In particular,
\begin{equation}\label{hardy}
\int_{0}^{R} \psi^{n-1} \xi_{r}^{2} \, dr 
\geq
\frac{(n-2)^2}{4} \int_{0}^{R}  \psi^{n-1} \frac{\xi^{2}}{\psi^2}   \, dr
\quad\textrm{for all radial }\xi \in C^1_0(\mathcal{B}_R).
\end{equation}
\end{Proposition}

Note inequality \eqref{improved:hardy} is really an improved Hardy inequality only if 
$H_{n,\psi}>0$. This holds for any geodesic ball in the hyperbolic space. Unfortunately, 
in the elliptic case we only have been able to prove it for geodesic balls of radius 
$R<R_0$. It would be interesting to obtain an improvement of the constant $H_{n,\psi}$ 
defined in \eqref{Hardy:ctant} to have an \eqref{improved:hardy} in large balls (with 
positive $H_{n,\psi})$.

%
%


Finally, let us to mention that the bibliography studying the regularity of extremal 
solutions in a general domain $\Omega\subset\mathbb{R}^n$ with the standard Euclidean 
metric is extensive. However, only partial answers are known for general nonlinearities 
$f$. We refer the reader to \cite{Cabre09,CS,Dupaigne,Nedev,Nedev01,S13,V13} and 
references therein.

\begin{Notation}
We always assume that the radius $R$ of the geodesic ball $\mathcal{B}_R$ is fixed. 
Therefore, all the universal constants appearing in this work, included the ones in the 
estimates of Theorem~\ref{thm1}, may depend on $R$. Moreover, as usual we denote by $C$ or $M$ 
the universal constants appearing in some inequalities in this paper. The value of these 
constants may vary even in the same line.
\end{Notation}

The paper is organized as follows. In Section \ref{radial} we prove the radial 
symmetry and the monotonicity property of semistable solutions established in 
Theorem~\ref{radsemistab}. Section~\ref{regradial} deals with the regularity 
of semistable and extremal solutions. We prove our $L^\infty$, $L^p$, and 
$W^{1,p}$ estimates of Theorem \ref{thm1} and Corollary~\ref{cor1}. 
Finally, in Section~\ref{examples} we find the extremal parameter 
and the extremal solution for the exponential and power nonlinearities 
considered in Theorems~\ref{thm:hyper} and \ref{thm:ellip}, establishing the sharpness of 
Theorem~\ref{thm1}.

\section{Radial symmetry of semistable solutions}\label{radial}

This section will be devoted to the proof of Theorem \ref{radsemistab}. 
The radial symmetry of positive solutions to uniformly elliptic problems on radially 
symmetric domains has been subject of an extensive study, essentially 
started by the celebrated work of Gidas, Ni, and Nirenberg \cite{GNN}. Most 
of these symmetry results are based on the moving plane method as well 
as on the use of the Maximum Principle and its generalizations. Here, we 
will follow a more direct approach which uses the semistability of our 
solutions and was applied in \cite{CES3} and \cite{CES4} to obtain 
symmetry results for semistable solutions to reaction-diffusion equations 
involving the $p$-Laplacian.

\begin{proof}[Proof of Theorem~{\rm\ref{radsemistab}}] 
Let $u\in C^2(\mathcal{B}_R)$ be a classical semistable solution of 
\eqref{probintro}. Note that the semistability condition \eqref{semistability} is
equivalent to the nonnegativity of the first eigenvalue of the linearized 
operator $-\Delta_g-f'(u)$ in $\mathcal{B}_R$, \textit{i.e.},
\begin{equation}\label{1st:linearized}
\lambda_1(-\Delta_g-f'(u),\mathcal{B}_R)
=
\inf_{\xi\in H^1_0(\mathcal{B}_R)\setminus \{0\}}
\frac{\int_{\mathcal{B}_R}\{|\nabla_g\xi|^2-f'(u)\xi^2\}\,dv_g}{\int_{\mathcal{B}_R}\xi^2\,dv_g}
\geq 
0.
\end{equation}

Let $u_\theta = \frac{\partial u}{\partial \theta}$ be any angular derivative 
of $u$. On the one hand, by the fact that $u\in C^2(\mathcal{B}_R)$, we clearly have
$$
\int_{\mathcal{B}_R} |\nabla_g u_\theta|^2 \, dv_g < \infty.
$$
Moreover, the regularity up the boundary of $u$ and the fact that $u = 0$ on $\partial \mathcal{B}_R$ trivially 
give that $u_\theta = 0$ on $\partial \mathcal{B}_R$. Hence, $u_\theta \in H^{1}_{0} (\mathcal{B}_R)$. 

On the other hand, noting that in the spherical coordinates given by \eqref{polar} the Riemannian 
Laplacian of $u = u(r, \theta_1,..,\theta_{n-1})$ is given by
$$
\Delta_{g} u 
= 
\frac{1}{\psi(r)^{n-1}} (\psi(r)^{n-1} u_r)_r + \frac{1}{\psi(r)^{2}} \Delta_{\mathbb{S}^{n-1}} u,
$$
where $\Delta_{\mathbb{S}^{n-1}}$ is the Riemannian Laplacian on the unit sphere $\mathbb{S}^{n-1}$, 
and by the radial symmetry of the weight $\psi$, we 
can differentiate problem \eqref{probintro} to see that $u_\theta$ 
(weakly) satisfies
$$
\left\{
\begin{array}{rcll}
-\Delta_{g} u_\theta &=& f'(u) u_\theta  &\mbox{in}\, \mathcal{B}_R, \\
u_\theta &=& 0 &\mbox{on}\, \partial \mathcal{B}_R.
\end{array}
\right.
$$
Therefore, multiplying the above equation on $\mathcal{B}_R$ and integrating by parts we have
$$
\int_{\mathcal{B}_R}|\nabla_g u_\theta|^2-f'(u)u_\theta^2\,dv_g=0,
$$
and hence, from \eqref{1st:linearized} (taking $\xi=u_\theta$ if necessary) it follows necessarily that either $|u_\theta|$ is a first 
positive eigenfunction of the linearized operator at $u$ in $\mathcal{B}_R$ or $u_\theta\equiv 0$. But by the periodicity of $u$ with respect to $\theta$ we
see that $u_\theta$ necessarily changes sign unless it is constant (equal to zero).
Thus $u_\theta \equiv 0$ for any $\theta \in S^{n-1}$, which means that 
$u$ is radial. 

Finally, if we pass to radial coordinates we see that $u = u(r)$ 
satisfies
$$
- \Big(\psi (r)^{n-1} u_r \Big)_r = \psi (r)^{n-1} f(u)  \qquad\mbox{in}\, (0,R).
$$
Integrating the previous equation from $0$ to any $s\in(0,R)$ with respect to $r$, recalling 
that $f(u)$ is positive, $\psi$ is also positive in $(0,R]$, and $u_r(0)=0$, we have
$$
\psi (s)^{n-1} u_r (s) 
= \int_{0}^{s} \Big( \psi (r)^{n-1} u_r (r) \Big)_r \, dr 
= - \int_{0}^{s} \psi (r)^{n-1} f(u(r)) \, dr 
<0.
$$
Thus $u_r(s) < 0$ for all $s\in(0,R)$, \textit{i.e.}, $u$ is decreasing. This concludes 
the proof. 
\end{proof}
\section{Regularity of radial semistable solutions}\label{regradial}

Let us begin by rewriting problem \eqref{probintro}, for radial solutions 
$u\in C^2(\mathcal{B}_R)$, as
\begin{equation}\label{probrad2}
\left\{
\begin{array}{rcll}
- \Big( \psi(r)^{n-1} u_r \Big)_r &=& \psi(r)^{n-1} f(u)  &\mbox{in}\, (0,R),\\
u &>& 0 &\mbox{in}\, (0,R),\\
u_r (0) = u (R) &=& 0,
\end{array}
\right.
\end{equation}
and considering the quadratic form associated to the second variation of the energy functional, 
evaluated at $u$, written in radial form:
$$
Q_u (\xi) := \int_{0}^{R} \psi(r)^{n-1} \{\xi_{r}^{2} -f'(u)\xi^2\}\, dr
$$
for every Lipschitz function $\xi$ such that $\xi(R)=0$.

We want to see that the results by Cabr\'e and Capella in \cite{CC05} 
for the Euclidean case carry over to the general Riemannian model setting. 
We start by proving the following lemma.

\begin{Lemma}
Let $f$ be a locally Lipschitz positive function. Assume that $\psi\in C^2([0,R])$ 
is positive in $(0,R]$ and satisfies \eqref{hyppsi}.
If $u\in C^2(\mathcal{B}_R)$ is a semistable classical solution of \eqref{probintro},
then
\begin{equation}\label{etapsi}
(n-1)\int_0^R\psi^{n-1}u_r^2(\psi')^2\eta^2\,dr
\leq
\int_0^R\psi^{n-1}u_r^2\{(\psi\eta)_r^2+(n-1)\psi\psi''\eta^2\}\,dr
\end{equation}
for every Lipschitz function $\eta$ such that $\eta(R)=0$.
\end{Lemma}
\begin{proof}
Differentiating equation \eqref{probrad2} it is easy to see that
\begin{equation}\label{ur}
- ( \psi^{n-1} u_{rr} )_r 
= 
\psi^{n-1} \left(f'(u)+(n-1)\left(\frac{\psi'}{\psi}\right)'\right) u_r
\quad\mbox{ in } (0,R).
\end{equation}

Thanks to equation \eqref{ur}, we are able prove that for any $\eta
\in H^1 \cap L^\infty (0,R)$ with support in $(0,R)$
there holds
$$
Q_u ( u_r \psi \eta ) = 
\int_{0}^{R} \psi^{n-1} u_{r}^{2} \left\{(\psi\eta)_{r}^{2} 
- 
(n-1)\Big((\psi')^2-\psi\psi''\Big)\eta^2\right\}\, dr
\geq 0.
$$

In fact, integrating by parts and using \eqref{ur} we are able to
compute

$$
\begin{array}{l}
Q_u (u_r \psi \eta)
=
\displaystyle 
\int_{0}^{R} \psi^{n-1} \left\{ u_{rr}^{2} \psi^2 \eta^2+ u_{r}^{2} ( \psi \eta )_{r}^{2} 
+ 
(\psi^2 \eta^2)_r u_{r} u_{rr} - f'(u) u_{r}^{2} \psi^2 \eta^2 \right\}  \, dr
\\
= 
\displaystyle 
\int_{0}^{R}  \left\{ u_{rr}^{2} \psi^2 \eta^2+ u_{r}^{2} ( \psi \eta )_{r}^{2} 
- f'(u) u_{r}^{2} \psi^2 \eta^2 \right\} \psi^{n-1} - \psi^2 \eta^2 (u_{r} u_{rr}\psi^{n-1})_r  \, dr
\\
= 
\displaystyle 
\int_{0}^{R}  \left\{ u_{r}^{2} ( \psi \eta )_{r}^{2} - f'(u) u_{r}^{2} \psi^2 \eta^2 \right\} \psi^{n-1} 
- \psi^2 \eta^2 u_{r} ( u_{rr}\psi^{n-1})_r  \, dr
\\
= 
\displaystyle 
\int_{0}^{R} \psi^{n-1} u_{r}^{2}\left\{( \psi \eta )_{r}^{2} 
+(n-1)\left(\frac{\psi'}{\psi}\right)'(\psi\eta)^2\right\}\, dr
\\
= 
\displaystyle 
\int_{0}^{R} \psi^{n-1} u_{r}^{2}\left\{( \psi \eta )_{r}^{2} 
-(n-1)\Big((\psi')^2-\psi\psi''\Big)\eta^2\right\}\, dr.
\end{array}
$$
Since for $n \geq 2$ we have that the singleton $\{0\}$ is of zero
capacity, the fact that $u \in H_{0}^{1} (\mathcal{B}_R)$ gives that
the equation above also holds for $\eta$ not necessarily vanishing
around $0$ with $|\nabla (\psi \eta)| \in L^\infty$.
\end{proof}

Now, we are able to prove the key estimate \eqref{keykey} used in  
our main regularity result.

\begin{Lemma}\label{lemma:key}
Let $f$ be a locally Lipschitz positive function. Assume that $\psi\in C^2([0,R])$ 
is positive in $(0,R]$ and satisfies \eqref{hyppsi}. Let $\delta=\delta(\psi)\in(0,R/2)$ 
be such that $\psi'>0$ in $[0,\delta]$.
If $u\in C^2(\mathcal{B}_R)$ is a semistable classical solution of \eqref{probintro}, 
then there exists a positive constant $C_{n,\alpha,\psi}$ depending only on $n$, $\alpha$, 
and $\psi$ such that
$$
\int_{0}^{\delta} u_{r}^{2} \psi^{n-1-2\alpha}\,dr 
\leq
C_{n,\alpha,\psi}\, \| u \|_{L^1 (\mathcal{B}_R)}^2
$$
for every
\begin{equation}\label{eq:alpha}
1 \leq \alpha < 1 + \sqrt{n-1}.
\end{equation}
\end{Lemma}
\begin{proof}
Let $\varepsilon\in(0,\delta)$ and define
$$
\eta_\varepsilon (r) 
:=\left\{ 
\begin{array}{cll}
\psi (\varepsilon)^{-\alpha} - \psi (\delta)^{-\alpha} \quad &\mbox{if } 0 \leq r \leq \varepsilon,
\\
\psi (r)^{-\alpha} - \psi (\delta)^{-\alpha} \quad &\mbox{if } \varepsilon \leq r \leq \delta,
\\
0 \qquad \quad  &\mbox{if }  \delta \leq r \leq R.
\end{array}
\right.
$$
Observe that both $\eta_\varepsilon$ and $(\psi\eta_\varepsilon)_r$ are bounded. 
By \eqref{etapsi} with $\eta=\eta_\varepsilon$ we obtain
$$
\begin{array}{l}
\displaystyle  
(n-1)\int_0^\varepsilon \psi^{n-1}u_r^2(\psi')^2\eta_\varepsilon^2\,dr
+
(n-1)\int_\varepsilon^\delta \psi^{n-1}u_r^2(\psi')^2\eta_\varepsilon^2\,dr\\
\displaystyle  
\leq 
\int_0^\varepsilon \psi^{n-1}u_r^2(\psi')^2\eta_\varepsilon^2\,dr
+
\int_\varepsilon^\delta\psi^{n-1}u_r^2(\psi')^2\Big\{(1-\alpha)\psi^{-\alpha}-\psi(\delta)^{-\alpha}\Big\}^2\,dr
\\
\displaystyle  
\hspace{0.4cm}+(n-1)\int_0^\delta\psi^{n-1}u_r^2\psi|\psi''|\eta_\varepsilon^2\,dr.
\end{array}
$$
Using that $n\geq 2$ and $\eta_\varepsilon^2\leq \psi^{-2\alpha}+\psi(\delta)^{-2\alpha}$, 
we have
$$
\begin{array}{ll}
\displaystyle  
(n-1)\int_\varepsilon^\delta \psi^{n-1}u_r^2(\psi')^2\eta_\varepsilon^2\,dr
&\leq 
\displaystyle  
\int_\varepsilon^\delta\psi^{n-1}u_r^2(\psi')^2\Big\{(1-\alpha)\psi^{-\alpha}-\psi(\delta)^{-\alpha}\Big\}^2\,dr
\\
&\displaystyle  
+(n-1)\int_0^\delta\psi^{n-1}u_r^2\psi|\psi''|(\psi^{-2\alpha}+\psi(\delta)^{-2\alpha})\,dr.
\end{array}
$$
Now, expanding and rearranging the terms in the integrals, and using that $\psi$ is 
increasing in $(0,\delta)$, we get

$$
\begin{array}{l}
\displaystyle  
(n-1-(1-\alpha)^2)\int_\varepsilon^\delta \psi^{n-1}u_r^2(\psi')^2\psi^{-2\alpha}\,dr
\\
\displaystyle  
\hspace{0.5cm}
\leq
\int_0^\delta\psi^{n-1}u_r^2(\psi')^2\psi^{-\alpha}\psi(\delta)^{-\alpha}
\Big\{2(\alpha+n-2)+\frac{\psi^\alpha}{\psi(\delta)^{\alpha}}\Big\}\,dr
\\
\displaystyle  
\hspace{0.5cm}
+(n-1)\int_0^\delta\psi^{n-1}u_r^2\psi|\psi''|\psi^{-2\alpha}
\Big\{1+\frac{\psi^{2\alpha}}{\psi(\delta)^{2\alpha}}\Big\}\,dr
\\
\displaystyle  
\hspace{0.5cm}
\leq
M_{n,\alpha,\psi}\int_0^\delta\psi^{n-1}u_r^2\psi^{-\alpha}\{(\psi')^2+\psi^{1-\alpha}\}\,dr,
\end{array}
$$
where $M_{n,\alpha,\psi}$ is a positive constant depending only on $n$, $\alpha$, and $\psi$.

Using that $\inf_{(0,\delta)}\psi'$ and $\sup_{(0,\delta)}\psi'$ are 
positive, \eqref{eq:alpha}, and letting $\varepsilon$ go to zero 
we get
\begin{equation}\label{key1}
\int_0^\delta \psi^{n-1}u_r^2\psi^{-2\alpha}\,dr
\leq
\frac{M_{n,\alpha,\psi}}{n-1-(1-\alpha)^2} \int_0^\delta \psi^{n-1}u_r^2\psi^{-\alpha}\left\{1+\psi^{1-\alpha}\right\}\,dr.
\end{equation}

Now, the fact that there exists a positive constant $C_{n,\alpha,\psi}$ 
depending only on $n$, $\alpha$, and $\psi$ such that
$$
\frac{M_{n,\alpha,\psi}}{n-1-(1-\alpha)^2}\, t^{-\alpha}(1+t^{1-\alpha})
\leq 
\frac{1}{2}t^{-2\alpha}+C_{n,\alpha,\psi}t^{n-1}
\quad 
\textrm{for all }t>0
$$
and \eqref{key1} give
\begin{equation}\label{key2}
\frac{1}{2}\int_0^\delta \psi^{n-1}u_r^2\psi^{-2\alpha}\,dr
\leq
C_{n,\alpha,\psi} \int_0^\delta \psi^{2n-2}u_r^2\,dr.
\end{equation}
Moreover, since $u$ is positive and radially decreasing (remember that $\delta\in(0,R/2)$ 
only depends on $\psi$), we have
\begin{equation}\label{udelta}
u(\delta) 
\leq 
C_{n,\psi} 
\int_{0}^{\delta} u (r) \psi^{n-1}\,dr 
\leq 
C_{n,\psi} \| u \|_{L^1 (\mathcal{B}_R)}
\end{equation}
and
\begin{equation}\label{ur:bis}
- u_r(\rho)
= - \frac{u(2\delta)-u(\delta)}{\delta}
\leq \frac{u(\delta)}{\delta}\quad
\textrm{for some }\rho\in(\delta,2\delta).
\end{equation}
Therefore, integrating the equation \eqref{probrad2} from $s\in(0,\delta)$ to
$\rho$ and noting that $f$ is positive, we obtain

$$
\begin{array}{lll}
- u_r(s) \psi(s)^{n-1} 
&=&\displaystyle 
- u_r(\rho) \psi(\rho)^{n-1} - \int_{s}^{\rho} f(u) \psi^{n-1}\,dr
\leq
\frac{u(\delta)}{\delta}\psi(\rho)^{n-1}
\\
&\leq& \displaystyle  C_{n,\psi} \| u \|_{L^1 (\mathcal{B}_R)}.
\end{array}
$$
Squaring this inequality and integrating for $s$ between $0$ and
$\delta$ we get
$$
\int_{0}^{\delta} u_{r}^{2} \psi^{2n-2}\,dr 
\leq 
C_{n,\psi} \| u \|_{L^1 (\mathcal{B}_R)}^2.
$$
We conclude the proof going back to \eqref{key2}.
\end{proof}


Thanks to Lemma \ref{lemma:key} we are now ready to give the proof of 
Theorem~{\rm\ref{thm1}}.

\begin{proof}[Proof of Theorem {\rm \ref{thm1}}]

Let $\delta\in(0,R/2)$ as in Lemma~\ref{lemma:key}. Using Schwarz inequality and \eqref{udelta}
we obtain
\begin{equation}\label{eq:6-21}
\begin{array}{lll}
\displaystyle |u(t)| 
&=&
\displaystyle 
\left|u(\delta)+\int _{t}^{\delta}-u_{r} \psi^{(n-1-2\alpha)/2} \psi^{(2\alpha-n+1)/2}dr\right|\nonumber
\\
&\le&  
\displaystyle 
C_{n,\psi} \|u\|_{L^1(\mathcal{B}_R)} 
+
\left( \int _{0}^{\delta} u_{r}^{2} \psi^{n-1-2\alpha}dr\right)^{\frac{1}{2}} 
\left(\int_{t}^{\delta}\psi^{2\alpha-n+1}dr\right)^{\frac{1}{2}}
\end{array}
\end{equation}
for all $t\in(0,\delta)$. Therefore, from Lemma~\ref{lemma:key} we deduce
\begin{equation}\label{eq:6-22}
|u(t)| 
\leq 
C_{n,\alpha,\psi} \left\{1+\left(\int_{t}^{\delta} \psi^{2\alpha-n+1}\,dr\right)^{\frac{1}{2}}\right\}
\| u\|_{L^1 (\mathcal{B}_R)}
\end{equation}
for all $t\in(0,\delta)$ and every $\alpha\in[1,1+\sqrt{n-1})$.

(a) \textit{$L^\infty$ estimate \eqref{Linfty:estimate}}: 
Assume $n\leq 9$. On the one hand, since $u$ is radially decreasing and thanks 
to \eqref{udelta}, we have that
\begin{equation}\label{linfty:bdary}
u(t) \leq u(\delta) \leq C_{n,\psi} \| u \|_{L^1 (\mathcal{B}_R)}
\quad
\textrm{for all }\delta \leq t < R.
\end{equation}

On the other hand, since $\psi\in C^2([0,R])$ is positive in $(0,R]$, $\psi(0)=0$, and 
$\psi'(0)=1$ by assumption, we note that the integral 
in \eqref{eq:6-22} is finite for $t=0$ if $2\alpha-n+1>-1$,
 \emph{i.e.},
$$
\int_0^{\delta} \psi^{2\alpha-n+1}dr\leq C_\psi<+\infty\qquad
\text{if} \qquad \alpha > \frac{n-2}{2}.
$$
Therefore, 
\begin{equation}\label{linfty:interior}
|u(t)| \leq C_{n,\alpha,\psi}\| u\|_{L^1 (\mathcal{B}_R)} \quad
\textrm{for all }0 < t < \delta,
\end{equation}
whenever 
$$
\max\left\{\frac{n-2}{2},1\right\}<\alpha<1+\sqrt{n-1}.
$$
Finally, since $2\leq n<10$, we can choose $\alpha$ (depending only on $n$) 
in the previous range to obtain \eqref{linfty:interior} with a constant 
$C_{n,\psi}$ depending only on $n$ and $\psi$.
The desired $L^\infty$ estimate \eqref{Linfty:estimate} follows 
from this fact and \eqref{linfty:bdary}.

(b) Assume $n\geq 10$.

\textit{$L^p$ estimate \eqref{Lp:estimate}}: On the one hand, the fact 
that $u$ is decreasing and \eqref{udelta} give that 
\begin{equation}\label{LqdeltaR}
\left(\int_\delta^R |u|^p\psi^{n-1}\,dt\right)^\frac{1}{p} 
\leq 
u(\delta) \left(\int_\delta^R \psi^{n-1}\,dt\right)^\frac{1}{p} \leq C_{n,\psi,p} \| u \|_{L^1 (\mathcal{B}_R)}.
\end{equation}

On the other hand, let $s\in(0,\delta)$. By \eqref{eq:6-22} it follows that
$$
\int_s^\delta |u|^p\psi^{n-1}\,dt
\leq
C_{n,\alpha,\psi}^p\| u\|_{L^1 (\mathcal{B}_R)}^p
\int_s^\delta \Big(1+\Big(\int_{t}^{\delta} \psi^{2\alpha-n+1}\,dr\Big)^{\frac{1}{2}}\Big)^p\psi^{n-1}\,dt
$$
for every $p\geq 1$. Notice that, again by \eqref{hyppsi}, we have:
$$
\int_0^\delta \Big(1+\Big(\int_{t}^{\delta} \psi^{2\alpha-n+1}\,dr\Big)^{\frac{1}{2}}\Big)^p\psi^{n-1}\,dt
\leq  C_{n,\alpha,\psi}< +\infty
$$
whenever
\begin{equation}\label{exppp}
\frac{2 \alpha - n + 2}{2} p + n-1 > -1,
\quad \textit{i.e.},\quad
p < \frac{2n}{n - 2 \alpha -2}.
\end{equation}
Therefore, for any 
$$
p < p_0=\frac{2n}{n - 2\sqrt{n-1}-4}
$$ 
we can choose $\alpha=\alpha(n,p)\in [1,1+\sqrt{n-1})$ such that condition \eqref{exppp}
holds, obtaining
%
$$
\left(\int_0^\delta |u|^p\psi^{n-1}\,dt\right)^\frac{1}{p}
\leq
C_{n,\psi,p}
\| u\|_{L^1 (\mathcal{B}_R)}.
$$
%
%
Taking into account \eqref{LqdeltaR} and applying Minkowski inequality, 
we reach the desired $L^p$ estimate \eqref{Lp:estimate}.

\textit{$W^{1,p}$ estimate \eqref{W1p:estimate}}:
Recall that every radial function $u$ in $H^1(\mathcal{B}_R)$ also 
belongs (as a function of $r = |x|$) to the Sobolev space $H^1(\delta, R)$ in one dimension.
Thus, by the Sobolev embedding in one dimension and \eqref{udelta}, we have
\begin{equation}\label{W1p:interior}
\begin{array}{lll}
\displaystyle \left(\int_\delta^R|u_r|^p\psi^{n-1}\,dr\right)^\frac{1}{p}
&\leq&
\displaystyle 
C_{n,\psi,p}\left(\int_\delta^R|u_r|^p\,dr\right)^\frac{1}{p}
\leq 
C_{n,\psi,p}\|u\|_{L^\infty(\delta,R)}
\\
&=&\displaystyle 
C_{n,\psi,p}u(\delta)
\leq 
C_{n,\psi,p}\|u\|_{L^1(\mathcal{B}_R)}.
\end{array}
\end{equation}
Observe that by equation \eqref{probrad2}, and since $f$ is positive, we have
$$
u_{rr}=-(n-1)\frac{\psi'}{\psi}u_r-f(u)\leq -(n-1)\frac{\psi'}{\psi}u_r
\quad \textrm{in }(0,R).
$$
Let $\rho\in(\delta,2\delta)$ such that \eqref{ur:bis} holds (as in the proof 
of Lemma~\ref{lemma:key}). Integrating the previous inequality with respect to $r$
from $t\in(0,\delta)$ to $\rho$, using \eqref{ur:bis} and \eqref{udelta}, 
as well as Schwarz inequality, we have
$$
\begin{array}{ll}
\displaystyle \frac{-u_r(t)}{n-1}
&\leq 
\displaystyle \frac{-u_r(\rho)}{n-1}+\int_t^\rho \frac{|\psi'|}{\psi}(-u_r)\,dr
\\
&\leq \displaystyle 
\frac{u(\delta)}{(n-1)\delta} +  \int_t^{2\delta} \frac{|\psi'|}{\psi}\psi^{-\frac{n-1}{2}+\alpha}(-u_r)\psi^{\frac{n-1}{2}-\alpha}\,dr
\\
&\leq \displaystyle 
C_{n,\psi}\|u\|_{L^1(\mathcal{B}_R)}
+ 
\left(\int_t^{2\delta} \left(\frac{\psi'}{\psi}\right)^2\psi^{-n+1+2\alpha}\,dr\right)^{\frac{1}{2}} 
\left(\int_t^{2\delta}u_r^2\psi^{n-1-2\alpha}\,dr\right)^{\frac{1}{2}}.
\end{array}
$$
Note that at this point we can use Lemma~\ref{lemma:key} with $\delta$ replaced by $2\delta$ 
(taking our original $\delta$ smaller if necessary). Using this fact 
we have
$$
-u_r(t)
\leq 
C_{n,\alpha,\psi}\|u\|_{L^1(\mathcal{B}_R)}
\Big(1+\Big(\int_t^{2\delta}\Big(\frac{\psi'}{\psi}\Big)^2\psi^{-n+1+2\alpha}\,dr\Big)^{\frac{1}{2}}\Big)
$$
for every $\alpha\in[1, 1 + \sqrt{n-1})$. Therefore, for this range of $\alpha$ and given $s\in (0,\delta)$, we get
$$
\int_s^\delta |u_r|^p\psi^{n-1}\,dt
\leq
C_{n,\alpha,\psi}^p\|u\|_{L^1(\mathcal{B}_R)}^p
\int_s^\delta\Big(1+\Big(\int_t^{2\delta} (\psi')^2\psi^{-n-1+2\alpha}\,dr\Big)^\frac{1}{2}\Big)^p\psi^{n-1}\,dt.
$$

Finally, note that 
$$
\int_0^\delta\Big(1+\Big(\int_t^{2\delta} (\psi')^2\psi^{-n-1+2\alpha}\,dr\Big)^\frac{1}{2}\Big)^p\psi^{n-1}\,dt
\leq C_{n,\psi,p}<+\infty
$$
whenever 
\begin{equation}\label{conddd}
\frac{2\alpha-n}{2} p + n-1 > -1,
\quad \textit{i.e.},\quad
p < \frac{2n}{n - 2 \alpha}
\end{equation}
(note that $n-2\alpha>0$ since $n\geq 10$ and $\alpha\in[1, 1 + \sqrt{n-1})$).
Therefore, for any 
$$
p < p_1=\frac{2n}{n- 2\sqrt{n-1}-2}
$$ 
we can choose $\alpha=\alpha(n,p)\in [1,1+\sqrt{n-1})$ such that \eqref{conddd}
holds, obtaining
$$
\int_s^\delta |u_r|^p\psi^{n-1}\,dt
\leq
C_{n,\psi,p}
\| u\|_{L^1 (\mathcal{B}_R)}^p.
$$

We conclude the proof using the previous estimate, \eqref{W1p:interior}, and Minkowski 
inequality, proving our $W^{1,p}$ estimate \eqref{W1p:estimate}.
\end{proof}

Finally, we prove Corollary~\ref{cor1} as an immediate consequence of Theorem~\ref{thm1}.

\begin{proof}[Proof of Corollary~{\rm\ref{cor1}}]
Since the extremal solution is a weak solution of the extremal problem \eqref{probla} 
for $\lambda=\lambda^*$, and hence $u^*\in L^1(\mathcal{B}_R)$, the result follows by 
applying Theorem~\ref{thm1} to minimal solutions $u_\lambda\in C^2(\mathcal{B}_R)$ for 
$\lambda\in(0,\lambda^*)$ and letting $\lambda\uparrow\lambda^*$.
\end{proof}


\section{Singular extremal solutions for exponential and power nonlinearities 
in space forms}\label{examples}

In this section we find the extremal parameter $\lambda^*$ and the extremal 
solution $u^*$ of problem \eqref{probla} for the exponential and power nonlinearities
considered in Theorems~\ref{thm:hyper} and \ref{thm:ellip}.

This will be achieved through the use of the Improved Hardy inequality established in 
Proposition~\ref{hardyprop} as well as the following 
uniqueness result, due to Brezis and V\'azquez \cite{BreVaz97}
for the Euclidean case (see also Proposition~3.2.1 in \cite{Dupaigne}). 
Its proof carries over easily to our setting thanks to the fact that, as commented in 
the Introduction, the structural hypothesis on the weight $\psi$ stated in 
\eqref{hyppsi} ensures that $\lambda_1(-\Delta_g;\mathcal{B}_R) > 0$.

\begin{Proposition}[\cite{BreVaz97,Dupaigne}]\label{uniqueness}
Let $\lambda_1(-\Delta_g;\mathcal{B}_R) > 0$ denote the principal eigenvalue
of the Dirichlet Laplace Beltrami operator $-\Delta_g$ in $\mathcal{B}_R$. 
Assume $f\in C^1(\mathbb{R})$ is convex. 

Let $u_1$, $u_2\in H^1_0(\mathcal{B}_R)$ be two stable weak solutions of \eqref{probintro}. 
Then, either $u_1 = u_2$ a.e. or $f(u) = \lambda_1 u$ on the essential ranges 
of $u_1$ and $u_2$. In the latter case, $u_1$ and $u_2$ belong to the eigenspace 
associated to $\lambda_1$. In particular, they are collinear.
\end{Proposition}

Let us prove the improved Hardy-type inequality on Riemannian models following the argument 
of Theorem~4.1 in \cite{BreVaz97}.

\begin{proof}[Proof of Proposition~{\rm \ref{hardyprop}}]
Let $\xi\in C^1_0(\mathcal{B}_R)$ be a radial function and let $\varphi:=\xi\psi^{\frac{n}{2}-1}$. 
We claim that the following Poincar\'e inequality holds:
\begin{equation}\label{poincare}
\int_0^R\varphi_r^2\psi\,dr\geq \frac{1}{4}(\sup_{(0,R)}(\phi/\psi))^{-2}\int_0^R\varphi^2\psi\,dr.
\end{equation}
Indeed, using integration by parts (note that $\varphi(0)=\varphi(R)=0$) and Schwarz inequality we have
$$
\begin{array}{lll}
\displaystyle \int_0^R\varphi^2\psi\,dr
&=&
\displaystyle \int_0^R\varphi^2\phi_r\,dr = -2\int_0^R \varphi\varphi_r\phi\psi^{1/2}\psi^{-1/2}\,dr
\\
&\leq&
\displaystyle 2\left(\int_0^R\varphi_r^2\psi\,dr\right)^{1/2}\left(\int_0^R\varphi^2\frac{\phi^2}{\psi^2}\psi\,dr\right)^{1/2}
\\
&\leq&
\displaystyle 2\sup_{(0,R)}(\phi/\psi)\left(\int_0^R\varphi_r^2\psi\,dr\right)^{1/2}\left(\int_0^R\varphi^2\psi\,dr\right)^{1/2}.
\end{array}
$$
The claim follows immediately from the previous inequality (note that 
$\sup_{(0,R)}(\phi/\psi)\in(0,+\infty)$ either for $\psi(r)=\sin r$, $r$, or $\sinh r$).

Now, using $(\psi')^2-1=-K_\psi\psi^2$, $\psi''/\psi=-K_\psi$,
and an integration by parts, we obtain
$$
\begin{array}{ll}
\displaystyle\int_0^R \left(\xi_r^2-\frac{(n-2)^2}{4}\frac{\xi^2}{\psi^2}\right)\psi^{n-1}\,dr
&\hspace{-0.3cm}=\hspace{-0.1cm}
\displaystyle
\int_0^R \varphi_r^2\psi-\frac{n-2}{2}(\varphi^2)_r\psi'-\frac{(n-2)^2}{4}K_\psi\varphi^2\psi\,dr
\\
&\hspace{-0.3cm}=\hspace{-0.1cm}
\displaystyle
\int_0^R \varphi_r^2\psi+\frac{n-2}{2}\left(\frac{\psi''}{\psi}-\frac{(n-2)}{2}K_\psi\right)\varphi^2\psi\,dr
\\
&\hspace{-0.3cm}=\hspace{-0.1cm}
\displaystyle
\int_0^R \varphi_r^2\psi-\frac{n(n-2)}{4}K_\psi\varphi^2\psi\,dr.
\end{array}
$$
We obtain \eqref{improved:hardy} using Poincar\'e inequality \eqref{poincare}.

Note that the constant $H_{n,\psi}$ defined in \eqref{Hardy:ctant}, for the hyperbolic and 
elliptic spaces, is given by
\begin{equation}\label{H:sinh}
H_{n,\sinh}=\frac{1}{4}\left(\frac{\sinh^2R}{(\cosh R-1)^2}+n(n-2)\right)
\end{equation}
and
\begin{equation}\label{H:sin}
H_{n,\sin}=\frac{1}{4}\left(\frac{\sin^2R}{(\cos R-1)^2}-n(n-2)\right),
\end{equation}
respectively. This constant is clearly positive for all $R$ in the hyperbolic space. 
Instead, in the elliptic space it is positive for all $R<R_0$ (by definition of 
$R_0$). Therefore inequality \eqref{hardy} is an immediate consequence of 
\eqref{improved:hardy}.
\end{proof}

We are now able to prove Theorems~\ref{thm:hyper} and \ref{thm:ellip} establishing the extremal 
parameter and the extremal solution of \eqref{probla} for the exponential and the power nonlinearities
defined in \eqref{exp:nonlinearity} and \eqref{power:nonlinearity} 
in the hyperbolic and elliptic spaces.

\subsection{Proof of Theorem~\ref{thm:hyper}~$(i)$ and Theorem~\ref{thm:ellip}~$(i)$: Exponential nonlinearity}

Consider problem \eqref{probla} with the exponential nonlinearity 
\begin{equation}\label{exp:nonl}
f(u) = \frac{e^{u}}{\psi(R)^2}-\frac{n-1}{n-2}K_\psi.
\end{equation}

It is clear that $f$ is a positive increasing nonlinearity satisfying 
\eqref{superlinear} in the hyperbolic space (since $K_\psi=-1$). 
Instead in the elliptic space these assumptions hold if and only if 
$$
R<R_{\rm e}=\sup\{s\in(0,\pi/2):\sin^2 s<\frac{n-2}{n-1}\}=\arcsin\left(\sqrt{\frac{n-2}{n-1}}\right).
$$
In these cases, as we said in the introduction, the minimal solution 
$u_\lambda\in C^2(\mathcal{B}_R)$ 
of \eqref{probla} exists for $\lambda\in(0,\lambda^*)$ and its increasing limit $u^*$ 
is a (weak) solution of the extremal problem \eqref{probla} for $\lambda=\lambda^*$.

A simple computation shows that problem \eqref{probla} admits the explicit singular solution 
$$
u^\#(r)=-2\log \left(\frac{\psi(r)}{\psi(R)}\right)\qquad \textrm{with }\lambda=\lambda^\#=2(n-2).
$$ 
Note that $u^\#\in H^1_0(\mathcal{B}_R)$ if $n\geq 3$.

We claim that $\lambda^\#=\lambda^*$ and $u^\#=u^*$ whenever $n\geq 10$ for any geodesic 
ball if $\psi=\sinh$ and for balls with radius $R<\min\{ R_0,R_{\rm e}\}$ if $\psi=\sin$.
Indeed, by Proposition~\ref{uniqueness} and since $u^\#\in H^1_0(\mathcal{B}_R)$ is singular at the origin,
we only have to prove that $u^\#$ is semistable. That is, 
\begin{equation}\label{semi:exp}
\int_{0}^{R} \psi^{n-1} \xi_{r}^{2} \, dr
\geq 
2(n-2)\int_0^R\psi^{n-1}\frac{\xi^2}{\psi^2}\,dr
\end{equation}
for every radial $\xi\in C^1_0(\mathcal{B}_R)$ (note that 
$\lambda^\#f'(u^\#)= 2(n-2)/\psi^2$). However, this inequality clearly holds 
by \eqref{hardy}:
$$
\int_{0}^{R} \psi^{n-1} \xi_{r}^{2} \, dr 
\geq
\frac{(n-2)^2}{4} \int_{0}^{R}  \psi^{n-1} \frac{\xi^{2}}{\psi^2}   \, dr
\quad\textrm{for all radial }\xi \in C^1_0(\mathcal{B}_R),
$$
since $(n-2)^2/4\geq 2(n-2)$ whenever $n\geq 10$.
This proves Theorem~\ref{thm:hyper}~$(i)$ and Theorem~\ref{thm:ellip}~$(i)$.

\subsection{Proof of Theorem~\ref{thm:hyper}~$(ii)$ and Theorem~\ref{thm:ellip}~$(ii)$: Power nonlinearity}
Consider now
\begin{equation}\label{power:nonl}
f(u) = (u+\psi(R)^{-\frac{2}{m-1}})\left((u+\psi(R)^{-\frac{2}{m-1}})^{m-1}
-\frac{(m-1)n-(m+1)}{(m-1)n-2m}K_\psi\right)
\end{equation}
with $m>(n+2)/(n-2)$ (\textit{i.e.}, $n>2(m+1)/(m-1)=2+4/(m-1)$). Note that
in part (ii) of Theorems~\ref{thm:hyper} and \ref{thm:ellip} we assume $n\geq N(m)$, 
where $N(m)$ is defined in \eqref{dim:powers}. In particular, one has $m>(n+2)/(n-2)$.

In the hyperbolic (and Euclidean) space it is clear that $f$ is a positive 
increasing nonlinearity satisfying \eqref{superlinear}. In the elliptic 
space these assumptions hold whenever
$$
f(0)=\sin^{-\frac{2}{n-1}}R\left(\sin^{-2}R-\frac{(m-1)n-(m+1)}{(m-1)n-2m}\right)>0,
$$
or equivalently, 
$$
\sin^{2}R<\frac{(m-1)n-2m}{(m-1)n-(m+1)}=:h(m,n).
$$
However, since the function $h$ defined in the right hand side of the above inequality 
is increasing in $m$ in $((n+2)/(n-2),+\infty)$, we have that $f(0)>0$ (independently 
of $m$) if 
$$
R<R_{\rm p}=\sup\left\{R\in(0,\pi/2):\sin^2R\leq\frac{n-2}{n}=h\left(\frac{n+2}{n-2},n\right)\right\}.
$$ 
Note that $R_{\rm p}$ coincides with the number defined in Theorem~\ref{thm:ellip}~(ii).

As a consequence, $f$ is a positive increasing nonlinearity satisfying \eqref{superlinear}
in all the space forms (whenever $R<R_{\rm p}$ in the elliptic one). Therefore, the minimal solution 
$u_\lambda$ of \eqref{probla} exists for $\lambda\in(0,\lambda^*)$ and its 
increasing limit $u^*$ is a (weak) solution of the extremal problem \eqref{probla} for 
$\lambda=\lambda^*$.

In order to find the extremal solution and the extremal parameter, let us note that 
$$
u^\#(r)=\psi(r)^{-\frac{2}{m-1}}-\psi(R)^{-\frac{2}{m-1}},\quad \textrm{with }
\lambda=\lambda^\#=\frac{2}{m-1}\left(n-\frac{2m}{m-1}\right),
$$
is a weak solution of \eqref{probla}. Note that, since $m>(n+2)/(n-2)$, 
we have $\lambda^\#>0$ and $u^\#\in H^1_0(\mathcal{B}_R)$.

We proceed as for the exponential nonlinearity, \textit{i.e.}, we want to prove 
that $u^\#$ is a semistable solution of \eqref{probla} for $\lambda=\lambda^\#$. 
First, note that
$$
f'(u^\#)=\frac{m}{\psi^2}-\frac{(m-1)n-(m+1)}{(m-1)n-2m}K_\psi.
$$
Therefore, semistability condition for $u^\#$ turns out to be
\begin{equation}\label{stability:powers}
\int_0^R \psi^{n-1}\xi_r^2\,dr
\geq
\lambda^\# m
\int_0^R \psi^{n-1}\left(\frac{1}{\psi^2}-\frac{1}{m}\frac{(m-1)n-(m+1)}{(m-1)n-2m}K_\psi\right)\xi^2\,dr 
\end{equation}
for every radial $\xi\in C^1_0(\mathcal{B}_R)$. By Proposition~\ref{hardyprop} we have
$$
\int_{0}^{R} \psi^{n-1} \xi_{r}^{2} \, dr 
\geq
\frac{(n-2)^2}{4} \int_{0}^{R}  \psi^{n-1} \frac{\xi^{2}}{\psi^2}   \, dr
+
H_{n,\psi}\int_0^R \psi^{n-1}\xi^2\,dr
$$
for every radial $\xi\in C^1_0(\mathcal{B}_R)$, where $H_{n,\psi}$ is the constant defined in 
\eqref{Hardy:ctant}. Therefore, semistability condition \eqref{stability:powers} 
follows from the previous improved Hardy inequality if the following two conditions 
hold:
\begin{equation}\label{cond:powers1}
\frac{(n-2)^2}{4} \geq \lambda^\# m = \frac{2m}{m-1}\left(n-\frac{2m}{m-1}\right)
\end{equation}
and
\begin{equation}\label{cond:powers2}
\begin{array}{lll}
H_{n,\psi}
&\geq& \displaystyle 
-\frac{2m}{m-1}\left(n-\frac{2m}{m-1}\right)
\frac{1}{m}\frac{(m-1)n-(m+1)}{(m-1)n-2m}K_\psi
\\
&\geq& \displaystyle 
-\frac{2}{(m-1)^2}\Big((m-1)n-(m+1)\Big)K_\psi.
\end{array}
\end{equation}
Note that condition \eqref{cond:powers1} is equivalent to 
\begin{equation}\label{cond3}
n\geq N(m)=2+\frac{4m}{m-1}+4\sqrt{\frac{m}{m-1}}.
\end{equation}

In order to deal with condition \eqref{cond:powers2} we consider the 
hyperbolic and the elliptic cases separately.

{\scshape Hyperbolic case}: Assume $\psi(r)= \sinh r$ and $K_\psi=-1$. 
We have (remember \eqref{H:sinh}) that condition \eqref{cond:powers2} is nothing but
$$
H_{n,\sinh}=\frac{1}{4}\left(\frac{\sinh^2 R}{(1-\cosh R)^2}+n(n-2)\right)
\geq
\frac{2}{(m-1)^2}\Big((m-1)n-(m+1)\Big). 
$$
It is clear that this inequality holds if 
$$
\frac{n(n-2)}{4}
\geq
\frac{2}{(m-1)^2}\Big((m-1)n-(m+1)\Big),
$$
or equivalently,
$$
n(n-2)(m-1)^2 \geq 8(m-1)(n-1)-16
$$
which is true whenever $m > \frac{n+2}{n-2}$. This shows that \eqref{cond:powers2}
holds independently of $R$ and therefore $u^\#$ is a semistable solution of 
\eqref{probla} for $\lambda=\lambda^\#$.

{\scshape Elliptic case}: Assume $\psi(r)=\sin r$ and $K_\psi=1$. In this case 
condition \eqref{cond:powers2} is
\begin{equation}\label{cond:powers3}
H_{n,\sin}=\frac{1}{4}\left(\frac{\sin^2 R}{(1-\cos R)^2}-n(n-2)\right)
\geq
- \frac{2}{(m-1)^2}\Big((m-1)n-(m+1)\Big)
\end{equation}
(rememeber \eqref{H:sin}). This condition clearly holds since we are assuming $R<R_0$, 
and hence, $H_{n,\sin}>0$. Therefore, in the elliptic case $u^\#$ is also a semistable 
solution.

We have thus obtained that $u^\#$ is a semistable solution of \eqref{probla} 
for $\lambda=\lambda^\#$ when \eqref{cond3} holds for any geodesic ball in the hyperbolic 
space and for geodesic balls of radius $R<\min\{R_0,R_{\rm p}\}$ in the elliptic one.  
Moreover, since it is singular at the origin, we obtain that $\lambda^\#=\lambda^*$ 
and $u^\#=u^*$ by Proposition~\ref{uniqueness}. This proves Theorem~\ref{thm:hyper}~$(ii)$ and
Theorem~\ref{thm:ellip}~$(ii)$ .

\bigskip
\noindent\textit{Acknowledgments.}
The authors would like to thank Xavier Cabr\'e for useful conversations on the topic of this paper.
The authors were supported by grants MINECO MTM2011-27739-C04 (Spain) and GENCAT 2009SGR-345 (Catalunya). 
The first author is also supported by PRIN09 project \textit{Nonlinear elliptic problems in the study 
of vortices and related topics} (Italy). The second author is also supported by ERC grant 320501 
(ANGEOM project). 




\begin{thebibliography}{99}


\bibitem{BFG}
Berchio, E., Ferrero, A., Grillo, G.: 
{Stability and qualitative properties of radial solutions of the Lane-Emden-Fowler equation on Riemannian models}.
Preprint: arXiv:1211.2762, to appear J. Math. Pure Appl..

\bibitem{BreCaMarRa96} 
Brezis, H.,  Cazenave, T., Martel Y., Ramiandrisoa, A.: 
{Blow up for $u_t-\Delta u=g(u)$ revisited}.
Adv. Differential Equations {\bf 1}, 73--90 (1996).

\bibitem{BreVaz97}
Brezis, H., V\'azquez, J.L.:
{Blow-up solutions of some nonlinear elliptic problems}.
Rev. Mat. Univ. Complut. Madrid {\bf 10}, 443--469 (1997).

\bibitem{Cabre09}
Cabr\'e, X.:
{Regularity of minimizers of semilinear elliptic problems up to dimension~4}.
Comm. Pure Appl. Math. {\bf 63}, 1362--1380 (2010).

\bibitem{CC05}
Cabr\'e, X., Capella, A.: 
{Regularity of radial minimizers and extremal solutions of semilinear elliptic equations}. 
J. Funct. Anal. {\bf 238}, 709--733 (2006).

\bibitem{CS}
Cabr\'e, X., Sanch\'on, M.:
{Geometric-type Sobolev inequalities and applications to the regularity of minimizers}.
J. Funct. Anal. {\bf 264}, 303--325 (2013). 

\bibitem{CES3} 
Castorina, D., Esposito, P., Sciunzi, B.:
{$p$-MEMS equation on a ball}. 
Methods Appl. Anal. {\bf 15}, 277--283 (2008).

\bibitem{CES4} 
Castorina, D., Esposito, P., Sciunzi, B.:
{Spectral theory for linearized $p$-Laplace equations}. 
Nonlinear Anal. {\bf 74}, 3606--3613 (2011).

\bibitem{CranRa75}
Crandall, M.G., Rabinowitz, P.H.: 
{Some continuation and variational methods for positive solutions of nonlinear elliptic
eigenvalue problems}. 
Arch. Ration. Mech. Anal. {\bf 58}, 207--218 (1975).

\bibitem{Dupaigne} 
Dupaigne, L.:
{Stable solutions to elliptic partial differential equations}. 
Monographs and Surveys in Pure and Applied Mathematics, 2011.

\bibitem{GNN}
Gidas, B., Ni, W.M., Nirenberg, L.:
{Symmetry and related properties via the maximum principle}.
Comm. Math. Phys. {\bf 68}, 209--243 (1979).

\bibitem{GS}
Grigor'yan, A., Saloff-Coste, L.:
{Stability results for Harnack inequalities}.
Ann. Inst. Fourier {\bf 55}, 825--890 (2005). 

\bibitem{JoLund73} 
Joseph, D.D., Lundgren, T.S.:
{Quasilinear Dirichlet problems driven by positive sources}.
Arch. Ration. Mech. Anal., {\bf 49}, 241--269 (1973).

\bibitem{KW}
Kazdan, J.L., Warner, F.W.: 
{Prescribing curvatures}.
Proceed. Symp. in Pure Math. {\bf 27}, 309--319 (1979).

\bibitem{MP80}
Mignot, F., Puel, J.P.:
{Sur une classe de probl\`emes non lin\'eaires avec
nonlin\'earit\'e positive, croissante, convexe}. 
Comm. Partial Differential Equations \textbf{5}, 791--836 (1980).

\bibitem{Nedev}
Nedev, G.:
{Regularity of the extremal solution of semilinear elliptic equations}.
C. R. Acad. Sci. Paris S\'er. I Math. {\bf 330}, 997--1002 (2000).

\bibitem{Nedev01}
Nedev, G.: 
{Extremal solution of semilinear elliptic equations}.
Preprint 2001.


\bibitem{S13} Sanch\'on, M.: 
{$W^{1,q}$ estimates for the extremal solution of reaction-diffusion problems}. 
Nonlinear Anal. {\bf 80}, 49--54 (2013).

\bibitem{V13} 
Villegas, S.: 
{Boundedness of extremal solutions in dimension 4}.
Adv. Math. {\bf 235}, 126--133 (2013).

\end{thebibliography}
\end{document}